\documentclass[11pt]{amsart}
\usepackage{mathptmx,amsmath,amssymb,enumerate,hyperref,tikz}   
\usetikzlibrary{arrows,positioning}
\calclayout

\renewcommand{\phi}{\ensuremath{\varphi}}
\newcommand{\Q}{\ensuremath{\mathbb{Q}}}
\newcommand{\Z}{\ensuremath{\mathbb{Z}}}
\newcommand{\C}{\ensuremath{\mathbb{C}}}
\newcommand{\A}{\ensuremath{\mathcal{A}}}
\newcommand{\QQ}{\ensuremath{\mathcal{Q}}}
\newcommand{\W}{\ensuremath{\mathcal{W}}}
\newcommand{\PP}{\ensuremath{\mathcal{P}}}
\newcommand{\into}{\ensuremath{\hookrightarrow}} 
\newcommand{\n}{\textbf{n}}
\DeclareMathOperator{\Hom}{Hom}				
\DeclareMathOperator{\End}{End}				
\DeclareMathOperator{\coker}{coker}
\DeclareMathOperator{\Lie}{Lie}

\DeclareMathOperator{\id}{id}

\DeclareMathOperator{\rk}{rk}				
\newcommand{\fiw}{FI$_\W$}
\newcommand{\M}{\mathfrak{M}}
\newcommand{\Cn}{\ensuremath{\mathcal{C}_n}}
\newcommand{\CC}{\ensuremath{\mathcal{C}}}
\newcommand{\An}{\ensuremath{\mathcal{A}_n}}

\DeclareMathOperator{\Ind}{Ind}

\newtheorem{lemma}{Lemma}[section]
\newtheorem{theorem}[lemma]{Theorem}
\newtheorem{proposition}[lemma]{Proposition}
\newtheorem{corollary}[lemma]{Corollary}
\theoremstyle{definition}
\newtheorem{example}[lemma]{Example}
\newtheorem{remark}[lemma]{Remark}

\begin{document}

\title[Representation stability of arrangements from root systems]{Representation stability for the cohomology of arrangements associated to root systems}
\author{Christin Bibby}
\address{Department of Mathematics, University of Western Ontario, London,
Canada}
\email{cbibby2@uwo.ca}

\begin{abstract}
From a root system, one may consider the arrangement of reflecting hyperplanes, as well as its toric and elliptic analogues. 
The corresponding Weyl group acts on the complement of the arrangement and hence on its cohomology. 
We consider a sequence of linear, toric, or elliptic arrangements which arise from a family of root systems of type A, B, C, or D, and
we show that the rational cohomology stabilizes as a sequence of Weyl group representations.
Our techniques combine a Leray spectral sequence argument similar to that of Church in the type A case along with FI$_W$-module theory which Wilson developed and used in the linear case.
A key to the proof relies on a combinatorial description, using labelled partitions, of the poset of connected components of intersections of subvarieties in the arrangement.
\end{abstract}

\maketitle

\section{Introduction}\label{s:intro}

In this paper, we consider arrangements of codimension-one subvarieties in a
complex vector space, torus, or abelian variety, determined by a root system of
type A, B, C, or D. The subvarieties in such an arrangement are determined by
realizing the root system as a set of characters on a torus. 
In each of these cases, the complement of the union of subvarieties comes with a
natural action of the corresponding Weyl group. 
This action makes the rational cohomology into a representation over the Weyl
group, which is the object we study. 

These arrangements arising from root systems also have interesting combinatorics. In the type A case, taking all intersections of subvarieties in the arrangement forms a lattice which is isomorphic to the partition lattice. In the other linear cases, Barcelo and Ihrig \cite{barceloihrig} give a combinatorial description of the intersection lattice. However, in the toric and elliptic cases, intersections of hyperplanes need not be connected and form a partially ordered set which is not necessarily a lattice. 
 In these cases, we consider the poset consisting of connected components of intersections, and in Theorem \ref{thm:components} we give a combinatorial description akin to that of Barcelo and Ihrig.
Understanding the Weyl group action on the connected components of intersections is then equivalent to understanding its action on certain types of partitions, called labelled partitions. 

We have already said that we are interested in the rational cohomology of the complement as a  representation. But more specifically, we consider the sequence of representations arising from each family of root systems. 
We show in Theorem \ref{thm:repstable} that this sequence of representations stabilizes in the sense of Church-Farb \cite{churchfarb}. That is, for $n$ large enough, if we decompose the representations into irreducibles, the multiplicity of each irreducible representation does not depend on $n$. 
As a special case, by taking the trivial representation, the orbit space enjoys homological stability (Corollary \ref{cor:homstab}).

In the case of the symmetric group, the complement is an ordered configuration space. Church \cite{church} showed representation stability of the rational cohomology of ordered configuration spaces using a Leray spectral sequence and the partition lattice. We generalize his method of using this spectral sequence for other types of arrangements by combining it with our combinatorics and with \fiw-module theory developed by Wilson \cite{wilson1,wilson2}. Wilson \cite{wilson1} also showed representation stability for each linear case.

We also give a slight improvement on Church's stable range for type A elliptic arrangements in Proposition \ref{prop:aelliptic}. 
Recently, Hersh and Reiner \cite{hershreiner} showed a better improvement for the type A linear case, and we wonder if their result or methods may also be applied to these other arrangements.  

\section{Arrangements}\label{s:arrangements}

\subsection{Linear, toric, and elliptic arrangements}\label{sec:arrbasics}

The three types of arrangements which we consider in this paper are as follows.
A \textit{linear arrangement} is a set of hyperplanes in a complex vector space,
A \textit{toric arrangement} is a set of codimension-one subtori (possibly
translated) in a complex torus. An \textit{abelian arrangement} is a set of
codimension-one abelian subvarieties (possibly translated) in a complex abelian variety.
In the case of an abelian arrangement, all of our abelian varieties will be products of an elliptic curve and we call it an \textit{elliptic arrangement}.
We denote the complement of $\A$ in $V$ by $M(\A)=V\setminus \cup_{H\in \A} H$. 

A \textit{layer} of an arrangement $\A$ is a connected component of an intersection $\bigcap_{H\in S} H$ for some subset $S\subseteq \A$. Note that the intersections themselves need not be connected. We say that the arrangement is \textit{unimodular} if every intersection is connected. 
The set of layers forms a ranked poset, ordered by reverse inclusion, with rank given by the complex codimension.
Note that linear arrangements are always unimodular, but this is not true in
general for toric or elliptic arrangements.

Locally, all of these arrangements look like classical hyperplane arrangements
(the linear case). We make this explicit here, as the notion of localization of
an arrangement is used in the proof of the key Lemma \ref{lem:E2}.
Let $F$ be a layer of a linear, toric, or elliptic arrangement $\A$.
For a point $p\in F$ not contained in any smaller layers of $\A$, define an arrangement 
$\A_F$ in the tangent space $T_pV$ consisting of hyperplanes $H_F:=T_pH$ for all
$H\supseteq F$. If $V$ has complex dimension $n$, then $\A_F$ is a central
hyperplane arrangement in $T_pV\cong\C^n$. This arrangement is referred to as
the \textit{localization} of $\A$ at $F$.
This arrangement, and its complement $M(\A_F)$, is independent of the choice of
p. That is, we may canonically identify the localizations at two different
generic points via translation.
The poset of layers of the arrangement $\A_F$
corresponds to taking the subposet of layers of $\A$ which contain $F$.

A natural way in which arrangements may arise is from a set of 
characters on a complex torus $T$, say $\Psi\subseteq\Hom(T,\C^\times)$. Here,
for each $\chi\in\Psi$, we take the set of connected components of
$\ker\chi\subseteq T$. This collection of subvarieties
defines a toric arrangement $\A(\C^\times,\Psi)$ in
$T$, in view of DeConcini and Procesi \cite{dcp}.
Noting that the Lie algebra $\Lie(T)$ is a complex vector space, we may take the
kernel of each $d\chi$ and get a linear arrangement $\A(\C,\Psi)$ in $\Lie(T)$.
Moreover, for a complex elliptic curve $E$, there is an embedding
$\Hom(\C^\times,\C^\times)\into \Hom(E,E)$ which sends the identity map on
$\C^\times$ to the identity map on $E$. This then extends to an embedding 
$$\iota:\Hom((\C^\times)^n,\C^\times)\into \Hom(E^n,E)$$
so that given $\Psi\subseteq\Hom(T,\C^\times)$, the collection of connected
components of $\ker\iota\chi$ in $E^n$ for $\chi\in \Psi$ gives an 
elliptic arrangement $\A(E,\Psi)$.

\subsection{Arrangements from root systems}

By taking the perspective of arrangements arising from characters on a torus, we
will now consider the case in which the set of characters is the set of positive roots in
a root system of type A, B, C, or D. 
Here, if $T$ is an $n$-dimensional torus, we will denote the root system by
$\Phi_n\subseteq\Hom(T,\C^\times)$. 
Letting $X$ be $\C$, $\C^\times$, or a complex elliptic curve, the collection of
connected components of the kernel of $d\chi$, $\chi$, or $\iota\chi$,
respectively, for $\chi\in\Phi_n^+$ gives an arrangement
which we denote by $\A(X,\Phi_n^+)$. 

Using the standard basis $v_1,\dots,v_n$ for the integer lattice $\Hom(T,\C^\times)$, 
the type C$_n$ root system consists of:
\[ \Phi_n = \{\pm(v_i\pm v_j)\ |\ 1\leq i<j\leq n\} \cup \{\pm 2v_i\ |\ 1\leq
i\leq n\} \]
and hence the positive roots are:
\[\Phi_n^+ = \{v_i\pm v_j \ |\ 1\leq i<j\leq n\} \cup \{2v_i\ |\ 1\leq i\leq n\} \]
For $\chi\in\Phi_n^+$, the kernel in $X^n$ is not necessarily connected. 
It is if $\chi=v_i\pm v_j$, giving the subvarieties:
\[H_{ij} := \{(x_1,\dots,x_n)\in X^n\ |\ x_i=x_j\}\] 
and
\[H'_{ij} := \{(x_1,\dots,x_n)\in X^n\ |\ x_i=x_j^{-1}\}\]
(writing the group operation on $X$ multiplicatively).
However, if $\chi=2v_i$, the connected components depend on $X$. 
Let $X[2]$ denote the two-torsion points of $X$. For $X=\C$, this consists only of
the origin, but $\C^\times$ has two two-torsion points and a complex elliptic curve has
four. The connected components of the kernel of $d\chi$, $\chi$, and $\iota\chi$ are then
indexed by $X[2]$:
\[H_i^z := \{(x_1,\dots,x_n)\in X^n\ |\ x_i=z\}\text{ for } z\in X[2]\]
In summary, the type C$_n$ arrangement in $X^n$, denoted by $\A(X,\Phi_n^+)$,
 is defined as the collection of
the above subvarieties: $H_{ij}$ (for $1\leq i<j\leq n$), $H'_{ij}$ (for $1\leq
i<j\leq n$), and $H_i^z$ (for $1\leq i\leq n$, $z\in X[2]$).

The type B$_n$ root system consists of:
\[\Phi_n = \{\pm(v_i\pm v_j)\ |\ 1\leq i<j\leq n\}\cup\{\pm v_i\ |\ 1\leq i\leq
n\}\]
Now, the kernel of $v_i$ (or similarly for $dv_i$ or $\iota v_i$) is the
identity component of the kernel of $2v_i$. 
Hence the type B$_n$ arrangement in $X^n$ consists of $H_{ij}$ ($1\leq i<j\leq
n$), $H'_{ij}$ ($1\leq i<j\leq n$), and $H_i^e$ ($1\leq i\leq n$, with $e$ the
identity of $X$). Note that in the linear case, the type B and C arrangements
are equal. 

The type D$_n$ root system consists of:
\[\Phi_n = \{\pm(v_i\pm v_j)\ |\ 1\leq i<j\leq n\}\]
hence the type D$_n$ arrangement consists of the subvarieties $H_{ij}$ and
 $H'_{ij}$ for $1\leq i<j\leq n$.

The type A$_{n-1}$ root system consists of:
\[\Phi_n = \{\pm(v_i-v_j)\ |\ 1\leq i<j\leq n\}\]
hence the type A$_{n-1}$ arrangement consists of the subvarieties $H_{ij}$ for
$1\leq i<j\leq n$.

Note that toric and elliptic arrangements of types B, C, and D, are not
unimodular, as $H_{ij}\cap H'_{ij}$ has connected components indexed by $X[2]$.
More specifically, $H_{ij}\cap H'_{ij}$ is the collection of points whose $i$'th
and $j$'th coordinates are both equal to each other and their inverse, hence
equal to a two-torsion point. 

\begin{example}
The best pictures we have for these arrangements are in $n=2$ with the real version of linear and toric arrangements. We draw here the pictures of the toric case, in $S^1\times S^1$; the subtori of the arrangement are the thickened lines. 
We warn the reader that, while the combinatorics of the real versus complex
pictures agree, the topology is very different. For example, the complement of
the complex arrangement is connected. 
\bigskip

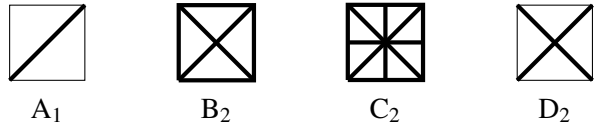
\begin{center}
\begin{tikzpicture}
\node at (.5,-.4) {A$_1$};
\draw[-] (0,0)--(1,0)--(1,1)--(0,1)--(0,0);
\draw[ultra thick,-] (0,0)--(1,1);
\end{tikzpicture}
\hspace{1cm}
\begin{tikzpicture}
\node at (.5,-.4) {B$_2$};
\draw[ultra thick,-] (0,0)--(1,0)--(1,1)--(0,1)--(0,0);
\draw[ultra thick,-] (0,0)--(1,1);
\draw[ultra thick,-] (1,0)--(0,1);
\end{tikzpicture}
\hspace{1cm}
\begin{tikzpicture}
\node at (.5,-.4) {C$_2$};
\draw[ultra thick,-] (0,0)--(1,0)--(1,1)--(0,1)--(0,0);
\draw[ultra thick,-] (0,0)--(1,1);
\draw[ultra thick,-] (1,0)--(0,1);
\draw[ultra thick,-] (0.5,0)--(0.5,1);
\draw[ultra thick,-] (0,0.5)--(1,0.5);
\end{tikzpicture}
\hspace{1cm}
\begin{tikzpicture}
\node at (.5,-.4) {D$_2$};
\draw[-] (0,0)--(1,0)--(1,1)--(0,1)--(0,0);
\draw[ultra thick,-] (0,0)--(1,1);
\draw[ultra thick,-] (1,0)--(0,1);
\end{tikzpicture}
\end{center}
\end{example}

\subsection{Weyl group action}

Let $X$ be one of $\C$, $\C^\times$, or a complex elliptic curve. Considering
our root system as a subset $\Phi\subseteq\Hom(X^n,X)$, the action of the
corresponding Weyl group $\W_n$ on $\Phi$ gives rise to a natural action of 
$\W_n$ on both the poset of layers and on the complement of the arrangement.
If $H_\chi$ is the kernel of $\chi\in\Phi$ and $w\in\W_n$, then $w\cdot H_\chi =
H_{w\cdot \chi}$. We will describe this explicitly in type C.

Consider the hyperoctahedral group $W_n=(\Z/2)\wr S_n = (\Z/2)^n\rtimes S_n$,
the Weyl group in types B$_n$ and C$_n$.
The group $W_n$ acts on $X^n$ via a combination of permuting the coordinates and
inverting some. More specifically, given 
$w=(\sigma,(\epsilon_1,\dots,\epsilon_n))\in S_n\ltimes (\Z/2)^n$ and 
$x=(x_1,\dots,x_n)\in X^n$, $w\cdot x$ has $\epsilon_i(x_i)$ as its $\sigma(i)$-th
coordinate.
Here, we are considering $\Z/2=\{\pm 1\}$ so that $\epsilon_i(x_i)=x_i$ if
$\epsilon_i=1$, and $\epsilon_i(x_i)=x_i^{-1}$ if $\epsilon_i=-1$.
This gives us the following action on our set of subvarieties:
\begin{itemize}
\item $w\cdot H_{ij} = H_{\sigma(i)\sigma(j)}$ if $\epsilon_i\epsilon_j=1$, 
\item $w\cdot H_{ij} = H'_{\sigma(i)\sigma(j)}$ if $\epsilon_i\epsilon_j=-1$,
\item $w\cdot H'_{ij} = H'_{\sigma(i)\sigma(j)}$ if $\epsilon_i\epsilon_j=1$, 
\item $w\cdot H'_{ij} = H_{\sigma(i)\sigma(j)}$ if $\epsilon_i\epsilon_j=-1$,
\item $w\cdot H_i^z = H_{\sigma(i)}^z$.
\end{itemize}
For this last subvariety, we note that since the two-torsion points are fixed by
the action of $W_n$ on $X$, this gives us the action on the connected components
of the kernel of the character $2v_i$, $\cup_{z\in X[2]} H_i^z$.
Denoting the kernel of $\chi\in \Phi$ by $H_\chi$, we also have the action on an
intersection given by $w\cdot\cap_{\chi\in S} H_\chi = \cap_{\chi \in S} w\cdot
H_\chi = \cap_{\chi\in S} H_{w\cdot \chi}$.
For a connected component of an intersection (ie, a layer), we see that if the
component has the point $z\in X[2]$ as the $i$-th coordinate, then the action
will send it to a layer whose $\sigma(i)$-th coordinate is $z$. This gives an
action of $W_n$ on the poset of layers.

Since the type B and D arrangements are subarrangements of the type C
arrangement, this also describes the action in these cases. For type A, we note
that we do not have an action of $\W_n$, but this description restricts to an
action of $S_n$. 

\begin{remark}\label{rmk:typeD}
While we naturally have an action of the corresponding Weyl group on our
arrangements, we emphasize that in the type D case, we actually have an action of the
hyperoctahedral group $W_n$. By
considering this action, rather than the type D Weyl group, on the type D$_n$
arrangements, we can get even
stronger results. By \cite[Prop.~3.22]{wilson2}, our stability result in Theorem
\ref{thm:repstable} stated for $W_n$ will imply stability for the type D Weyl
group, as stated in Corollary \ref{cor:restrict}.
Moreover, by working with a single group, $W_n$, to get our results in the 
three new cases (B, C, and D), this simplifies the exposition as well.
Thus, except for our Corollaries \ref{cor:restrict}, \ref{cor:polychar}, and \ref{cor:homstab},
 we will work exclusively with $W_n$
for the type B, C, and D arrangements, and with $S_n$ for the type A
arrangement.
\end{remark}

\section{Combinatorial description of layers}\label{s:components}

The goal of this section is provide a combinatorial description of layers, which
allows us to better understand and handle the Weyl group action on layers. It
will also help us to break down representations which appear in Section \ref{s:stability} into simpler building blocks, which we can then use to show stability. 
We start by describing the combinatorial objects needed and then show their relationship to layers in Theorem \ref{thm:components}. 

\subsection{Labelled partitions}

Let $S$ and $L$ be sets. We say that a \textit{partition of $S$ labelled by $L$}, or an $L$-labelled partition of $S$, is a partition $\Sigma$ of $S$, together with a subset $T\subseteq\Sigma$ and injection $f:T\to L$. 
We say that the parts in $T$ are the \textit{labelled parts} of $\Sigma$, while $\Sigma\setminus T$ consists of the \textit{unlabelled parts}. For $p\in T$ corresponding to $z\in L$, we say that $p$ is \textit{labelled by} $z$ and use the notation $p=\Sigma_z$. 
We use the convention that if $z$ is not in the image of $f$, then $\Sigma_z=\emptyset$. 
One may similarly define a partition of a number $k$ labelled by $L$. 

Given one of our root systems, we will define a particular set of labelled partitions, which in Theorem \ref{thm:components} will be shown to describe the layers of the corresponding arrangement.
First, we introduce some notation. 
Again let $X$ be $\C$, $\C^\times$, or a complex elliptic curve with 2-torsion points $X[2]$.
Let $[n]=\{1,2,\dots,n\}$ and $\n=\{1,\bar{1},2,\bar{2},\dots,n,\bar{n}\}$. For $S\subseteq \n$, let $\bar{S} = \{\bar{x}\ |\ x\in S\}$, taking $\bar{\bar{k}}=k$. 
We say that a set $S\subseteq\n$ is \textit{bar-invariant} (or self-barred as in
\cite{barceloihrig}) if $\bar{S}=S$.

Let $\PP_n(X)$ be the set of partitions $\Sigma$ of $\n$ labelled by $X[2]$ such that 
\begin{enumerate}[(i)]
\item for every $S\in \Sigma$, $\bar{S}\in \Sigma$, and 
\item $S=\bar{S}$ if and only if $S$ is labelled.
\end{enumerate}
Each type of root system $\Phi_n$ will correspond to a subset
$\CC(X,\Phi_n)\subseteq\PP_n(X)$ as follows:
\begin{itemize}
\item If $\Phi_n$ is type C, let $\CC(X,\Phi_n) = \PP_n(X)$.
\item If $\Phi_n$ is type B, take the subset of all $\Sigma\in\PP_n(X)$ such
that if $|\Sigma_z|=2$ then $z=e$ (the identity of $X$).
\item If $\Phi_n$ is type D, take the subset of all $\Sigma\in\PP_n(X)$ such
that $|\Sigma_z|\neq2$ for any $z\in X[2]$.
\item If $\Phi_n$ is type A, take the subset of all
$\Sigma=\Sigma_+\cup\overline{\Sigma_+}$ where $\Sigma_+\vdash[n]$.
\end{itemize}

In each type, $\CC(X,\Phi_n)$ is a partially ordered set, with $\Sigma<\Sigma'$ if $\Sigma$ is a refinement of $\Sigma'$ such that $\Sigma_z\subseteq\Sigma'_z$ for each $z\in X[2]$.
That is, we order it by refinements which respect the labelling. 
Moreover, $\CC(X,\Phi_n)$ is a ranked poset with $\rk(\Sigma) = n-\frac{\ell}{2}$, where $\ell$ is the number of unlabelled parts of $\Sigma$. 

\begin{example}\label{ex:partposetTC}
Below is the Hasse diagram of $\CC(\C^\times,\Phi_2)$ when $\Phi_2$ is a
type C root system. 
Note that the subscripts on some blocks in the partitions denote the
labelling; here, our two-torsion points are $X[2]=\{\pm 1\}$.

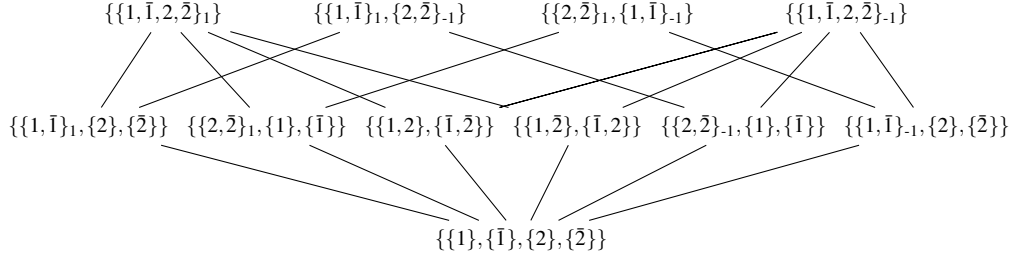
\begin{center}
{\scriptsize
\begin{tikzpicture}%[scale=0.1]
\node (0) at (0,0) {$\{\{1\},\{\bar{1}\},\{2\},\{\bar{2}\}\}$}; 
%{$1|\bar{1}|2|\bar{2}$};
\node[above left= 1cm and -1cm of 0] (4) {$\{\{1,2\},\{\bar{1},\bar{2}\}\}$}; 
% {$1,2|\bar{1},\bar{2}$};
\node[right=.3mm of 4] (5)  {$\{\{1,\bar{2}\},\{\bar{1},2\}\}$};
 %{$1,\bar{2}|\bar{1},2$};
\node[right=.3mm of 5] (1)  {$\{\{2,\bar{2}\}_{\text{-1}},\{1\},\{\bar{1}\}\}$}; 
%{$2,\bar{2}_{\text{-1}}|1|\bar{1}$};
\node[right=.3mm of 1] (6)   {$\{\{1,\bar{1}\}_{\text{-1}},\{2\},\{\bar{2}\}\}$}; 
%{$1,\bar{1}_{\text{-1}}|2|\bar{2}$};
\node[left=.3mm of 4] (3)  {$\{\{2,\bar{2}\}_1,\{ 1\},\{\bar{1}\}\}$};
%{$2,\bar{2}_1| 1|\bar{1}$};
\node[left=.3mm of 3]  (2)  {$\{\{1,\bar{1}\}_1,\{ 2\},\{\bar{2}\}\}$}; 
%{$1,\bar{1}_1| 2|\bar{2}$};
\node [above left= 1cm and 4mm of 5] (A)
{$\{\{1,\bar{1}\}_1,\{2,\bar{2}\}_{\text{-1}}\}$}; 
%{$1,\bar{1}_1|2,\bar{2}_{\text{-1}}$};
\node [above right= 1cm and 4mm of 4] (D)
{$\{\{2,\bar{2}\}_1,\{1,\bar{1}\}_{\text{-1}}\}$};
%{$2,\bar{2}_1|1,\bar{1}_{\text{-1}}$};
\node[right=of D] (C)  {$\{\{1,\bar{1},2,\bar{2}\}_{\text{-1}}\}$};
\node[left=of A] (B)  {$\{\{1,\bar{1},2,\bar{2}\}_1\}$};
\draw[-] (A)--(1)--(0)--(2)--(A);
\draw[-] (2)--(B)--(3)--(0)--(4)--(B)--(5);
\draw[-] (1)--(C)--(4)--(C)--(5)--(0)--(6)--(C);
\draw[-] (3)--(D)--(6);
\end{tikzpicture}}
\end{center}
\end{example}

\begin{example}\label{ex:partposetEB}
Below is the Hasse diagram of $\CC(E,\Phi_2)$ where $\Phi_2$ is a type B
root system and $E$ is a complex elliptic curve. Denote the two-torsion points
of $E$ by $z_1=e$ (the identity), $z_2$, $z_3$, and $z_4$.
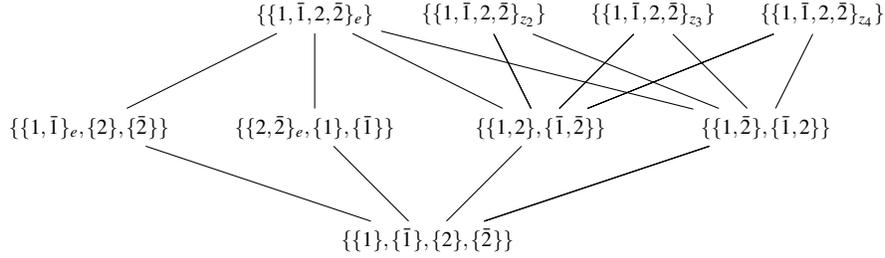
\begin{center}
{\scriptsize
\begin{tikzpicture}[scale=1.5]
\node (0) at (0,0) {$\{\{1\},\{\bar{1}\},\{2\},\{\bar{2}\}\}$};
\node (4) at (1,1) {$\{\{1,2\},\{\bar{1},\bar{2}\}\}$};
\node (5) at (3,1)  {$\{\{1,\bar{2}\},\{\bar{1},2\}\}$};
\node (3) at (-1,1) {$\{\{2,\bar{2}\}_{e},\{ 1\},\{\bar{1}\}\}$};
\node (2) at (-3,1) {$\{\{1,\bar{1}\}_{e},\{ 2\},\{\bar{2}\}\}$};
\node (C) at (2,2) {$\{\{1,\bar{1},2,\bar{2}\}_{z_3}\}$};
\node (B) at (.5,2) {$\{\{1,\bar{1},2,\bar{2}\}_{z_2}\}$};
\node (A) at (-1,2)  {$\{\{1,\bar{1},2,\bar{2}\}_e\}$};
\node (D) at (3.5,2) {$\{\{1,\bar{1},2,\bar{2}\}_{z_4}\}$};
\draw[-] (0)--(2)--(A)--(3)--(0)--(4)--(A)--(5);
\draw[-] (C)--(4)--(C)--(5)--(0);
\draw[-] (D)--(4)--(D)--(5)--(0);
\draw[-] (B)--(4)--(B)--(5)--(0);
\end{tikzpicture}}
\end{center}
\end{example}

\subsection{Weyl group action}
In type A, the set $\CC(X,\Phi_n)$ is really just the partition lattice, which
has an action of $S_n$. In the other types, we have an action of the
hyperoctahedral group $W_n=(\Z/2)\wr S_n$. 
This action is induced by its action on $\n$, but we will describe it more explicitly. 
Let $w=(\sigma,\epsilon)\in W_n$ with $\sigma\in S_n$ and $\epsilon\in(\Z/2)^n$. 
Then for $k\in [n]$, we have $w\cdot k = \overline{\sigma(k)}$ if
$\epsilon_k=-1$ and $w\cdot k= \sigma(k)$ if $\epsilon_k=1$. This extends to
$\n$ by $w\cdot \bar{k} =\overline{w\cdot k}$. 
Then for $S\subseteq \n$, we have $w\cdot S = \{w\cdot x\ |\ x\in S\}$, and for
a partition $\Sigma\vdash \n$, we have $w\cdot \Sigma= \{w\cdot S\ |\ S\in \Sigma\}$, 
If $\Sigma$ is a labelled partition, then $w\cdot \Sigma$ is labelled so that
$(w\cdot\Sigma)_z=w\cdot(\Sigma_z)$ for each label $z$.

Given a partition $\Sigma\in \PP_n(X)$,  we can define a partition $\widehat{\Sigma}$ of $n$ labelled by $X[2]$, where $\widehat{\Sigma}_z=\frac{|\Sigma_z|}{2}$ and the unlabelled parts are given by $|S|$ for each pair of unlabelled parts $S,\bar{S}\in \Sigma$. 
For example, if $\Sigma$ is the labelled partition 
\[ \{\{1,\bar{1}\}_e,\{2,\bar{4}\},\{\bar{2},4\},\{3\},\{\bar{3}\}\} \] of
\textbf{4}, then we have $\widehat{\Sigma} = (1_e,2,1)$, a labelled partition of
4.

\begin{lemma}\label{lem:orbits}
Let $\PP_n(X)$ be the set of partitions of $\n$ labelled by $X[2]$, and let
$\QQ_n(X)$ be the set of partitions of $n$ labelled by $X[2]$. 
The nonempty fibers of the function $\PP_n(X)\to \QQ_n(X)$ defined by 
$\Sigma\mapsto\widehat{\Sigma}$ are $W_n$-orbits.
\end{lemma}
\begin{proof}
The $W_n$--action preserves the size of the parts in the partition, which means
that for each $w\in W_n$ and $\Sigma\in\PP_n(X)$, we have
$\widehat{w\Sigma}=\widehat{\Sigma}$.
Moreover, if we have $\widehat{\Sigma}=\widehat{\Sigma'}$, then we could find
some $w\in W_n$ with $w\Sigma=\Sigma'$, constructed as follows. 

Suppose that
$\widehat{\Sigma}=(\lambda_{z_1},\dots,\lambda_{z_m},\lambda_1,\dots,\lambda_\ell)$,
where $X[2]=\{z_1,\dots,z_m\}$. 
So we have $|\Sigma_{z_i}|=|\Sigma'_{z_i}|$ for each $i$; let $\sigma_i$ be a
permutation in $S_n$ which gives a one-to-one correspondence between these
two parts, and let $w_i=(\sigma_i,1)$.
Now, since $\lambda_1,\dots,\lambda_\ell$ denotes the sizes of (pairs of)
unlabelled blocks, we can index the unlabelled parts of $\Sigma$ and $\Sigma'$
as $S_1,\bar{S}_1,\dots,S_\ell,\bar{S}_\ell$ and
$S'_1,\bar{S'}_1,\dots,S'_\ell,\bar{S'}_\ell$ so that we have one-to-one
correspondences between $S_i$ and $S'_i$. Ignoring the bars, this one-to-one
correspondence determines a permutation $\sigma_i\in S_n$. 
To get $w_i=(\sigma_i,\epsilon_i)$, we let the $j$-th coordinate of $\epsilon_i$
be $-1$ if the one-to-one correspondence either sends $j$ to $\bar{k}$ or
$\bar{j}$ to $k$ for some $k$; otherwise the $j$-th coordinate of $\epsilon_i$
will be 1.
With this $w=w_{z_1}\cdots w_{z_m}w_1\cdots w_\ell$, we have $\Sigma'=w\Sigma$.
\end{proof}

This lemma means that for each $\lambda\in\QQ_n(X)$, the set
$\{\Sigma\in \CC(X,\Phi_n)\ |\ \widehat{\Sigma}=\lambda\}$ is either a 
$W_n$-orbit or is empty. 
Example \ref{ex:decomp} depicts the orbits of $\CC(\C^\times,\Phi_2)$ when
$\Phi_2$ is a type B root system.

\subsection{Layers as labelled partitions}\label{sec:components}

The goal of this section is to prove that these labelled partitions give a combinatorial description of layers of the arrangement. 
This description will help us to get a handle on certain representations in Lemma \ref{lem:E2}.
But before proving our claim in Theorem \ref{thm:components}, we consider some
examples to demonstrate it.

\begin{example}\label{ex:compposetTC}
Below is the Hasse diagram for the poset of layers in the case that
$\Phi_2$ is a type C root system and $X=\C^\times$. The bijection with the poset
in Example \ref{ex:partposetTC} should be visible here. 
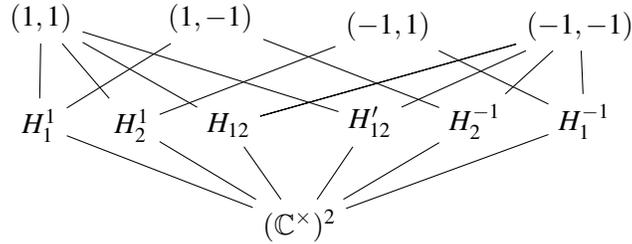
\begin{center}
\begin{tikzpicture}[scale=0.45]
\node (0) at (0,0) {$(\C^\times)^2$}; 
\node[above left = 7mm and -1mm of 0] (4)  {$H_{12}$};
\node[above right= 7mm and -1mm of 0] (5)   {$H'_{12}$};
\node[right= 5mm of 5] (1)  {$H_2^{-1}$};
\node[right=5mm of 1] (6)   {$H_1^{-1}$};
\node[left=5mm of 4] (3)  {$H_2^1$};
\node[left=5mm of 3]  (2)  {$H_1^1$};
\node [above left= 7mm and 1cm of 5] (A)   {$(1,-1)$};
\node [above right= 7mm and 1cm of 4] (D)  {$(-1,1)$};
\node[right=of D] (C) {$(-1,-1)$}; 
\node[left=of A] (B)  {$(1,1)$};
\draw[-] (A)--(1)--(0)--(2)--(A);
\draw[-] (2)--(B)--(3)--(0)--(4)--(B)--(5);
\draw[-] (1)--(C)--(4)--(C)--(5)--(0)--(6)--(C);
\draw[-] (3)--(D)--(6);
\end{tikzpicture}
\end{center}
\end{example}

\begin{example}\label{ex:compposetEB}
Below is the Hasse diagram of the poset of layers in the case that $\Phi_2$
is a type B root system and $X=E$ is a complex elliptic curve. 
Denote the two-torsion points
of $E$ by $z_1=e$ (the identity), $z_2$, $z_3$, and $z_4$.
The bijection with the poset in Example \ref{ex:partposetEB} should be visible.
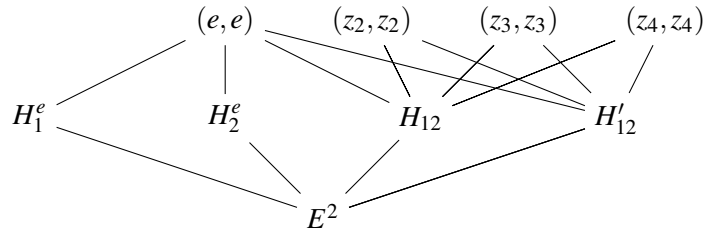
\begin{center}
\begin{tikzpicture}[scale=1.3]
\node (0) at (0,0) {$E^2$};
\node (4) at (1,1) {$H_{12}$};
\node (5) at (3,1)  {$H'_{12}$};
\node (3) at (-1,1) {$H_2^e$};
\node  (2) at (-3,1)  {$H_1^e$};
\node (B) at (.5,2) {$(z_2,z_2)$}; 
\node (A) at (-1,2) {$(e,e)$}; 
\node (C) at (2,2) {$(z_3,z_3)$}; 
\node (D) at (3.5,2) {$(z_4,z_4)$}; 
\draw[-] (0)--(2)--(A)--(3)--(0)--(4)--(A)--(5);
\draw[-] (C)--(4)--(C)--(5)--(0);
\draw[-] (B)--(4)--(B)--(5)--(0);
\draw[-] (D)--(4)--(D)--(5)--(0);
\end{tikzpicture}
\end{center}
\end{example}

We finally describe how the poset of labelled partitions $\CC(X,\Phi_n)$
corresponds to the poset of layers of the arrangement $\A(X,\Phi_n^+)$.
Let $X$ be one of $\C$, $\C^\times$, or a complex elliptic curve, let 
$\Phi_n$ be a root system of type B, C, or D, and let $W_n$
be the hyperoctahedral group.

Given a partition $\Sigma\in\CC(X,\Phi_n)$, we will define a layer $F_\Sigma$ of
the arrangement $\A(X,\Phi_n^+)$ as follows. For $S\in\Sigma$, take the
collection of subvarieties $H_{ij}$ if either $i,j\in S$ or $\bar{i},\bar{j}\in
S$; $H'_{ij}$ if either $i,\bar{j}\in S$ or $\bar{i},j\in S$; and $H_i^z$ if
$S=\Sigma_z$ and $i\in S$. Denote the intersection of these subvarieties by
$F_S$, so that we define:
\[ F_\Sigma = \bigcap_{S\in\Sigma} F_S \]
Our claim, in Theorem 1, is that $F_\Sigma$ is indeed a layer of the
arrangement and that this gives a bijection between the two posets.
Since $F_\Sigma$ makes an appearance later in the paper, we will also give a more
explicit description of it now. 

Write
$\Sigma=\{S_1,\bar{S_1},\dots,S_\ell,\bar{S_\ell},\Sigma_{z_1},\dots,\Sigma_{z_m}\}$
where $\{z_1,\dots,z_m\}\subseteq X[2]$.
For $S\in \Sigma$, let $X_S$ denote the factor of $X^n$ corresponding to indices
$i$ with either $i\in S$ or $\bar{i}\in S$, so that we can write:
\[X^n = X_{S_1}\times\cdots\times X_{S_\ell}\times X_{\Sigma_{z_1}}\times\cdots\times
X_{\Sigma_{z_m}} \]
We have inclusions $\iota_k:X\into X_{S_k}$ such that the $i$-th coordinate of
$\iota_k(x)$ is $x$ when $i\in S_k$ and $x^{-1}$ when $\bar{i}\in S_k$.
We also have the inclusion $\iota_{z_j}:\{z_j\}\into X_{\Sigma_{z_j}}$ whose
image is the point $(z_j,\dots,z_j)$.
For a single $S_k$, we have an inclusion $X\times X^{n-|S_k|}\to X_{S_k}\times
X^{n-|S_k|}$ given by $\iota_k\times\id$ whose image is $F_{S_k}$.
Similarly, for $S=\Sigma_{z_j}$, $\iota_{z_j}\times\id$ gives an inclusion
$\{z_j\}\times X^{n-|S|}\to X_{\Sigma_{z_j}}\times X^{n-|S|}$ whose image is
$F_{\Sigma_{z_j}}$.

Together,
$\iota_1\times\cdots\times\iota_\ell\times\iota_{z_1}\times\cdots\times\iota_{z_m}$
gives an inclusion 
\[X\times\cdots\times X\times\{z_1\}\times\cdots\times\{z_m\} \to 
 X_{S_1}\times\cdots\times X_{S_\ell}\times X_{\Sigma_{z_1}}\times\cdots\times
X_{\Sigma_{z_m}}\] with image $F_\Sigma$. 
In particular, observe that $F_\Sigma$ is connected and has codimension
$n-\ell$, where $\ell$ is the number of pairs of unlabelled parts in $\Sigma$.

We demonstrate one more example before proving that this assignment
$\Sigma\mapsto F_\Sigma$ does indeed give a bijection.

\begin{example}\label{ex:tripleint}
Consider a triple intersection of hyperplanes
$H_{ij}$, $H'_{ij}$ and $H_{ik}$ ($i,j,k$ distinct) in $X^n$ (for $n\geq 3$).
We have $$H_{ij}\cap H'_{ij}\cap H_{ik} = \{(x_1,\dots,x_n)\in X^n \ |\ x_i=x_j,
x_i=x_j^{-1}, x_i=x_k\}$$
So for a point $x$ in the intersection, the $i$-th, $j$-th, and $k$-th
coordinates must all be equal and be in $X[2]$. So we can write:
$$H_{ij}\cap H'_{ij}\cap H_{ik} = \bigcup_{z\in X[2]} \{(x_1,\dots,x_n)\in X^n\ |\
x_i=x_j=x_k=z\}.$$
If $\{i,j,k\}=\{1,2,3\}$, the connected components will correspond to the 
labelled partitions  $$\Sigma =
\{\{1,\bar{1},2,\bar{2},3,\bar{3}\}_z,\{4\},\{\bar{4}\},\dots,\{n\},\{\bar{n}\}\}$$
with $z\in X[2]$.
Note in this example, setting $\{i,j,k\}=\{1,2,3\}$ and $z\in X[2]$, we have 
$F_\Sigma$ as a connected component of $H_{ij}\cap H'_{ij}\cap H_{ik}$, but this can also be written as the
intersection of all of these subvarieties: 
$H_{12}, H_{13}, H_{23}, H'_{12}, H'_{13}, H'_{23}, H_1^z, H_2^z, H_3^z$.
\end{example}

\begin{theorem}\label{thm:components}
Let $\Phi_n$ be a root system of type B, C, or D, and let $X$ be one of $\C$,
$\C^\times$, or a complex elliptic curve. 
Denote the corresponding arrangement by $\An=\A(X,\Phi_n^+)$ and the
corresponding set of labelled partitions by $\Cn=\CC(X,\Phi_n)$.
Then there is a $W_n$-equivariant isomorphism of ranked posets between $\Cn$ and the layers of $\An$. 
\end{theorem}
\begin{proof}
We will first describe the bijection and prove the theorem for type C, and then we restrict to our other types.

\textit{The bijection in type C:}

Let $\Sigma\in\PP_n(X)$, our poset of type C, and consider $F_\Sigma$ as
described above. It is an intersection of subvarieties in $\An$ and it is
connected, thus it is a layer of the arrangement.
Moreover, if $\Sigma$ and $\Sigma'$ are distinct elements of $\PP_n(X)$, then
$F_\Sigma$ and $F_{\Sigma'}$ are distinct subvarieties in $X^n$. To see this,
suppose that we have $F_\Sigma=F_{\Sigma'}$. 
Suppose we had $\Sigma_z\neq \Sigma'_{z}$ for some $z\in X[2]$. If, say, $i\in
\Sigma_z\setminus \Sigma'_{z}$, then every point of $F_\Sigma$ would have the
$i$-th coordinate equal to $z$ while this is not the case in $F_{\Sigma'}$.
So we must have $\Sigma_z=\Sigma'_z$ for each $z$. 
If the unlabelled parts differed, then we would have some $i,j$ such that either $i$
and $j$, or $i$ and $\bar{j}$, are in the same unlabelled part in $\Sigma$ but
not in $\Sigma'$ (or vice versa).
This means that for every $x\in F_\Sigma$ we have $x_i=x_j$, or $x_i=x_j^{-1}$
respectively, while this is not true of every point in $F_{\Sigma'}$. 

Now let $F$ be a layer of $\An$, and we will define a labelled partition
$\Sigma$ in $\PP_n(X)$ such that $F=F_\Sigma$. 
From $F$, we can define an equivalence relation on $\n$ as follows:
$i\sim j$ iff $\bar{i}\sim \bar{j}$ iff $H_{ij}\supseteq F$, 
$i\sim\bar{j}$ iff $\bar{i}\sim j$ iff $H'_{ij}\supseteq F$, and
$i\sim \bar{i}$ iff $H_i^z\supseteq F$ for some $z$. 
This gives (by taking equivalence classes) a partition of $\n$ where some parts
satisfy $S=\bar{S}$ and the others come in pairs $(S,\bar{S})$. We will label
the bar-invariant parts so that we get an element of $\PP_n(X)$. 
If $S=\bar{S}$, then for each $i\in S$, there exists a $z\in X[2]$ with $H_i^z\supseteq F$. Moreover, this $z$ is the same no matter which $i\in S$ we consider: if $i$ and $j$ are both in $S$, then $i\sim j$ and $i\sim \bar{j}$, which means $F$ is contained in a connected component of $H_{ij}\cap H'_{ij}$ (the one corresponding to our $z$). Thus, we may label $S$ by $z$.
Moreover, if $S$ and $T$ are distinct bar-invariant parts, they correspond to different elements of $X[2]$; otherwise, we'd have for $i,\bar{i}\in S$ and $j,\bar{j}\in T$ such that $H_{ij}\supseteq F$ and hence $i\sim j$. 

\textit{Compatibility with order and rank:} (Type C)

Assume that $\Sigma$ is a refinement of $\Sigma'$ such that $\Sigma_z\subseteq\Sigma'_z$ for all $z\in X[2]$. 
For all $S\in\Sigma$, there exists $T\in \Sigma'$ such that $S\subseteq T$, and all $T\in \Sigma'$ have such an $S$. Moreover, if $S=\Sigma_z$, then $T=\Sigma'_z$. Since $S\subseteq T$ implies $F_S\supseteq F_T$, we have
\[F_\Sigma = \bigcap_{S\in\Sigma} F_S\supseteq \bigcap_{T\in\Sigma'} F_T =
F_{\Sigma'}.\]

As for rank, recall from our construction that the codimension of $F_\Sigma$ is
equal to $n-\ell$, when $\Sigma$ has $2\ell$ unlabelled parts. This is equal to
the rank of $\Sigma$. 

\textit{Compatibility with Weyl group action:} (Type C)

Let $w=(\sigma,\epsilon)\in W_n$ and $F$ a layer of $\An$. 
If $H_{ij}\supseteq F$, then $$wF \subseteq wH_{ij} = \begin{cases}
H_{\sigma(i)\sigma(j)} & \text{ if } \epsilon_i\epsilon_j=1\\
H'_{\sigma(i)\sigma(j)} & \text{ if } \epsilon_i\epsilon_j=-1\\
\end{cases}$$
Similarly, if $H'_{ij}\supseteq F$, then $$wF\subseteq wH'_{ij} = \begin{cases}
H'_{\sigma(i)\sigma(j)} & \text{ if } \epsilon_i\epsilon_j=1\\
H_{\sigma(i)\sigma(j)} & \text{ if } \epsilon_i\epsilon_j=-1\\
\end{cases}$$
Finally, if  $H_i^z\supseteq F$, then $H_{\sigma(i)}^z\supseteq wF$. 
The first two pieces imply that if $S\in\Sigma$ is unlabelled then $wS$ is an unlabelled part of $w\Sigma$, and the last one implies that  $(w\Sigma)_z = w(\Sigma_z)$.

\textit{Type B$_n$:}

We now have $\Cn$ the set of $\Sigma\in\PP_n(X)$ such that if $|\Sigma_z|=2$
then $z=e$, and $\An$ now denotes the type B$_n$ arrangement in $X^n$.
Given $\Sigma\in\Cn$, we may construct $F_\Sigma$ as above, but we need to show
that $F_\Sigma$ is a layer of $\An$.
The type B$_n$ arrangement is a subarrangement of type C$_n$, where we exclude
$H_i^z$ for $z\neq e$. It is clear that if $S\in\Sigma$ is unlabelled, then $F_S$ is a layer; we need only worry about $F_{\Sigma_z}$. 
If $|\Sigma_z|=2$, then $z=e$, and we have $F_{\Sigma_z} = H_i^e$. 
If $|\Sigma_z|\neq2$, then consider the intersection $H_{\Sigma_z}$ of the subvarieties $H_{ij}$ and $H'_{ij}$ for $i,j\in \Sigma_z$. This intersection is not connected, but its connected components are indexed by $X[2]$, and $F_{\Sigma_z}$ is the connected component indexed by $z$.

We also need to show that in the inverse map, if we are restricting ourselves to
layers of the type B$_n$ arrangement, the partition we get will not have
$|\Sigma_z|=2$ for $z\neq e$. 
Suppose that $|\Sigma_z| = 2$; then there exists $i$ such that $i\sim \bar{i}$ but no $j$ with $i\sim j$ or $i\sim \bar{j}$. This implies that $H_i^z\supseteq F$ but no $H_{ij}$ or $H'_{ij}$ contains $F$. The only way this can be a layer in type B is if $z=e$.

\textit{Type D$_n$:}

Now let $\Cn$ be the set of $\Sigma\in\PP_n(X)$ such that $|\Sigma_z|\neq2$ for any $z\in X[2]$. 
Given such $\Sigma$, we may again construct $F_\Sigma$ as above, but we need to show that this is a layer of the type D$_n$ arrangement. 
As in type B, we need only worry about $F_{\Sigma_z}$ being a layer. But since $|\Sigma_z|$ is never 2, we will have $F_{\Sigma_z}$ as a connected component of the intersection $H_{\Sigma_z}  = \bigcap_{i,j\in\Sigma_z}(H_{ij}\cap H'_{ij})$.

On the other hand, suppose that we have a layer $F$ of the type D arrangement
and construct the corresponding partition $\Sigma\in\PP_n(X)$. If $\Sigma_z=\{i,\bar{i}\}$, then there is no $j$ such that $H_{ij}$ or $H'_{ij}$ contains $F$, contradicting the fact that $F$ is a layer.
\end{proof}

\begin{remark}
The analogous statement for type A$_{n-1}$ is clear, because the poset $\Cn$ in this case is equivalent to the partition lattice of the set $[n]$. 
\end{remark}

\begin{remark}
In the linear case, our description is equivalent to that given by Barcelo and Ihrig \cite[Theorems 3.1\&4.1]{barceloihrig}.
They showed that the poset in question is also isomorphic to the lattice of parabolic subgroups of the Weyl group.
It is also worth noting that in the type B/C linear case, this is the Dowling lattice. But in other cases, this labelling helps us take into account the more complicated structure of having multiple connected components of intersections. 
\end{remark}

\section{Representation stability}\label{s:stability}

Our goal is to show representation stability for the cohomology of our arrangements, but we first briefly review representation stability and its main tool of \fiw-modules. 
Throughout this section, we let $\W_n$ denote either the symmetric group $S_n$ (type A) or the hyperoctahedral group $W_n$ (type B/C). 
For more details on the theory, we refer the reader to \cite{cef,churchfarb} for the case of $S_n$ (and much more) and \cite{wilson2,wilson1} for the case of $W_n$ (as well as the type D Weyl group).

Note that we are working over characteristic zero throughout this paper. 
A representation of a group $G$ will always mean a group homomorphism 
$G\to\operatorname{GL}(V)$ where $V$ is a finite-dimensional vector
space over $\Q$. 
Unless otherwise stated, cohomology will always be with rational coefficients,
and we will write $H^*(X)$ to mean $H^*(X;\Q)$ for a space $X$. 

\subsection{$\W_n$-representation stability}

To discuss representation stability for a sequence of groups, one needs a consistent way of describing the irreducible representations. 
There are many cases in which this can be done, including the classical families of Weyl groups. 

For the symmetric group $S_n$, irreducible representations are indexed by partitions of $n$. If we consider a partition $\lambda=(\lambda_1,\dots,\lambda_\ell)$ of $k$ with $\lambda_1\geq\cdots\geq\lambda_\ell>0$ and $n\geq \lambda_1+k$, we may write $V(\lambda)_n$ to denote the irreducible representation of $S_n$ indexed by the partition $\lambda[n]:=(n-k,\lambda_1,\dots,\lambda_\ell)$.
For example, in this notation, $V(0)_n$ is always the trivial representation and $V(1)_n$ is always the standard representation. 

For the hyperoctahedral group $W_n$, irreducible representations are indexed by pairs of partitions $\lambda=(\lambda^+,\lambda^-)$ where $|\lambda^+|+|\lambda^-|=n$. Given a pair of partitions $\lambda=(\lambda^+,\lambda^-)$, where $\lambda^-$ is a partition of $k$, and $n$ large enough, we may write $V(\lambda)_n$ to be the irreducible representation of $W_n$ corresponding to $(\lambda^+[n-k],\lambda^-)$. 
For example, $V(0,0)_n$ is always the trivial representation. 

We start with a \textit{consistent} sequence $\{V_n\}$ of $\W_n$-representations; that is, each $V_n$ is a $\W_n$-representation along with $\W_n$-equivariant maps $\phi_n:V_n\to V_{n+1}$. 
Such a sequence is said to be \textit{uniformly representation stable} with stable range $n\geq N$ if for $n\geq N$\dots 
\begin{enumerate}
\item the map $\phi_n$ is injective,
\item the image $\phi_n(V_n)$ generates $V_{n+1}$ as a $\Q[\W_{n+1}]$-module, and 
\item $V_n=\displaystyle\bigoplus_\lambda c_{\lambda}V(\lambda)_n$, where the multiplicities $c_\lambda$ do not depend on $n$.
\end{enumerate}

\subsection{\fiw-modules}

Consider the category \fiw\ (where $\W$ denotes either type A or type B/C) consisting of objects $\n$ (with $\textbf{0}=\emptyset$) and morphisms $f:\textbf{m}\to\n$ which are injections such that $f(\bar{k})=\overline{f(k)}$ for all $k\in\textbf{m}$, also requiring that $f([n])\subseteq [n]$ if $\W$ is type A. 
An \textit{\fiw-module} is a functor $V$ from the category \fiw\ to the category of $\Q$-modules. We denote by $V_n$ the image of $\n$. Since $\End(\n)=\W_n$ in the category \fiw, the $\Q$-module $V_n$ comes equipped with an action of $\W_n$. Moreover, the natural inclusions $\n\into \textbf{n+1}$ induce $\W_n$-equivariant maps $V_n\to V_{n+1}$, making the sequence $\{V_n\}$ a consistent sequence of $\W_n$-representations. 

A map of \fiw-modules is a natural transformation. We say $U$ is a \textit{sub-\fiw-module} of $V$ if there is a map $U\to V$ such that $U_n$ is a subrepresentation of $V_n$ for all $n$. 
An \fiw-module $V$ is \textit{finitely generated} if there is a finite set of elements of $\sqcup V_n$ that are not contained in any proper sub-\fiw-module.
An \fiw-module $V$ has \textit{stability degree} $\leq s$ if $(V_{n+a})_{\W_n}\cong (V_{n+1+a})_{\W_{n+1}}$ for every $a\geq0$ and $n\geq s$, where the subscript denotes the coinvariants. 
We say that $V$ has \textit{weight} $\leq d$ if for all $n$, every irreducible representation $V(\lambda)_n$ appearing with nonzero multiplicity in $V_n$ satisfies $|\lambda|\leq d$ (if $\lambda$ is a partition) or $|\lambda^+|+|\lambda^-|\leq d$ (if $\lambda=(\lambda^+,\lambda^-)$ is a pair of partitions). 
Again, we refer the reader to \cite{wilson2,wilson1} for more details on these concepts; we state here the main properties and example on which our results rely. 

We start by stating a proposition on finitely generated FI$_\W$-modules, bounding
the weight and stability degree for kernels, cokernels, and extensions. 
The statements on stability degree were made in \cite[Lemma~6.3.2.]{cef} for
type A. Wilson extended this to FI$_\W$ to establish part (1) in
\cite[Prop.~4.18]{wilson2}, and the same argument gives part (2).
The statement on weight in part (1) follows from Definition 4.1 in
\cite{wilson2}, and part (2) for weights follows from semisimplicity of the
representations (since we are in characteristic zero).

\begin{proposition}\label{prop:fiwprops}
\ 
\begin{enumerate}
\item Assume that $f:U\to V$ is a map of \fiw-modules which are finitely generated with weight $\leq d$ and stability degree $\leq s$. Then $\ker(f)$ and $\coker(f)$ are both finitely generated with weight $\leq d$ and stability degree $\leq s$. 
\item Assume that $0\to U\to V\to Q\to 0$ is a short exact sequence of \fiw-modules, where $U$ and $Q$ are both finitely generated with weight $\leq d$ and stability degree $\leq s$. 
Then so is $V$. 
\end{enumerate}
\end{proposition}

Now, since finitely generated FI$_\W$-modules form an abelian category,
a spectral sequence of finitely generated FI$_\W$-modules converges to a
finitely generated FI$_\W$-module. 
See, for example, \cite[Cor. 2.5]{kupersmiller} or \cite[Thm.
3.3]{jimenezrolland} in the type A case.
However, it takes a little work to get a bound on the stability degree of the abutment. See,
for example, \cite[Thm. 6.3.1]{cef} for an argument on the bounds for weight and
stability degree in the case of configuration spaces. One could make the same
kind of argument for the spaces which we work with, but we state a formulation
which could be applied more generally. We state it for sequences which have
$E_3=E_\infty$, but it could (with a bit more book-keeping) be stated for
sequences with $E_r=E_\infty$. The idea is that Proposition
\ref{prop:fiwprops} tells us how to use the weight and stability degree on one
page to bound the weight and stability degree on the next, continuing until the
sequence collapses. We benefit from the fact that all of our sequences collapse
early.

\begin{proposition}\label{prop:ssfiwmods}
Suppose that $E_*^{pq}$ is a first quadrant spectral sequence of \fiw--modules
which converges to the \fiw-module $H^{p+q}$, and assume that $E_3^{pq}=E_\infty^{pq}$. 
If $E_2^{pq}$ is finitely generated with weight $\leq d_{p+2q}$ and stability
degree $\leq s_{p+2q}$, then $H^i$ is finitely generated with weight $\leq
\max\limits_{0\leq p\leq i}\{d_{2i-p}\}$ and stability degree $\leq 
\max\limits_{0\leq p\leq i}\{s_{2i-p}\}$. 
\end{proposition}
\begin{proof}
Since each of the FI$_\W$-modules in a sequence $E_2^{p-2,q+1}\to E_2^{p,q}\to
E_2^{p+2,q-1}$ has weight $\leq d_{p+2q}$ and stability degree $\leq s_{p+2q}$,
part (1) of Proposition \ref{prop:fiwprops} implies that
$E_\infty^{p,q}=E_3^{p,q}$ has weight $\leq d_{p+2q}$ and stability degree $\leq
s_{p+2q}$.

Fix $i$, and consider the filtration $F_0\subseteq\cdots\subseteq F_i = H^i$,
where $F_j/F_{j-1} = E_\infty^{j,i-j}$. Since $F_0$ and $F_1/F_0$ are both
finitely generated, so is $F_1$ by part (2) of Proposition \ref{prop:fiwprops},
with weight $\leq \max\{d_{2i},d_{2i-1}\}$ and stability degree $\leq\max\{s_{2i},s_{2i-1}\}$. Repeating this for each $F_{j-1}$ and $F_j/F_{j-1}$ ($j=1,2,\dots,i$) gives that $H^i$ is finitely generated with the desired bounds on weight and stability degree. 
\end{proof}

Now having the bounds on weight and stability degree is what allows us to get a
stable range for our sequences, as the following theorem says, due to
Church-Ellenberg-Farb in type A and Wilson in type B/C.

\begin{theorem}\cite[Thm. 4.26]{wilson2},\cite[Thm. 2.58]{cef}\label{thm:fgstable}
If $V$ is an \fiw-module, with $\W$ of type A or B/C, which is finitely generated with weight $\leq d$ and stability degree $\leq s$, then the sequence $\{V_n\}$ of $\W_n$-representations with maps $V_n\to V_{n+1}$ induced by the natural inclusions \emph{$\n\to\textbf{n+1}$} is uniformly representation stable with stable range $n\geq d+s$. 
\end{theorem}

In the following example, we describe particularly nice \fiw-modules, which will be useful for us in the next section. 

\begin{example} \label{ex:mw} (\cite[Ex. 1.5.5]{wilson1})
Let $U$ be a $\W_k$-representation which is finite dimensional. Consider the
\fiw-module $\M_\W(U)$ which takes $\n$ to $0$ if $n<k$ and otherwise to the $\W_n$-representation $\Ind_{\W_k\times \W_{n-k}}^{\W_n} U\boxtimes \Q$, where $U\boxtimes\Q$ is the external tensor product of $U$ with the trivial $\W_{n-k}$-representation $\Q$. 
$\M_\W(U)$ is a finitely generated \fiw-module with weight $\leq k$ and stability degree $\leq k$. 
Thus, 
the sequence of induced representations $$\left\{\Ind_{\W_k\times\W_{n-k}}^{\W_n} U\boxtimes\Q\right\}$$
is representation stable with stable range $n\geq 2k$. 
\end{example}

\subsection{Arrangements associated to root systems}
Let $\Phi_n$ be a root system of type B, C, or D, and let $X$ be one of $\C$,
$\C^\times$, or a complex elliptic curve. Then recall our notation of
$\An=\A(X,\Phi_n^+)$ for the corresponding arrangement, with complement $M(\An)$
in $X^n$. 
Again, by Remark \ref{rmk:typeD}, we will work with the action of the
hyperoctahedral group $W_n$ until the corollaries at the end of this section.

Consider the Leray spectral sequence of the inclusion $f:M(\An)\into X^n$, which is given by
$$E_2^{pq}(n) = H^p(X^n;R^qf_*\Q) \implies H^{p+q}(M(\An)).$$
Our goal is to show representation stability of the cohomology, and so we start by understanding the $E_2$-term as a representation. 

\begin{lemma}\label{lem:E2}
Let $\Phi_n$ be a root system of type B, C, or D, and let $X$ be one of $\C$, 
$\C^\times$, or a complex elliptic curve. Consider the corresponding arrangement 
$\A(X,\Phi_n^+)$ and Leray spectral sequence. Assume that $p,q\geq0$ and $n\geq p+2q$.

There are $W_k$--representations $V(\lambda,r,\alpha)$ indexed by some finite
set $I=\{(\lambda,r,\alpha)\}$, where $k\leq p+2q$ depends on
$(\lambda,r,\alpha)$, such that
$$E_2^{pq}(n) = \bigoplus_I \Ind_{W_k\times W_{n-k}}^{W_n} V(\lambda,r,\alpha)\boxtimes\Q.$$
\end{lemma}
\begin{proof}
Fix the notation of $\Phi_n$, $X$, $\An=\A(X,\Phi_n)$, $\Cn=\CC(X,\Phi_n)$, and $\QQ_n=\QQ_n(X)$.
Also assume that $p,q\geq0$ and $n\geq p+2q$.
Recall that throughout, cohomology is assumed to have rational coefficients.

We start with a known decomposition \cite[Lemma 3.1]{bibby} as follows, which we write using our description of layers from Theorem \ref{thm:components}:
$$E_2^{pq}(n) \cong \bigoplus_{\Sigma} H^p(F_\Sigma) \otimes H^q(M(\A_{F_\Sigma}))$$
where the sum is taken over all $\Sigma\in\CC(X,\Phi_n)$ such that $\rk(\Sigma)=q$.
We recall that $F_\Sigma$ denotes the layer of $\An$ corresponding to the partition
$\Sigma$, as in Section \ref{sec:components}, and $\A_{F_\Sigma}$ denotes the
localization of the arrangement $\An$ at $F_\Sigma$, as discussed in Section
\ref{sec:arrbasics}.

The action of $W_n$ on $X^n$ induces the representation on 
$E_2^{pq}(n)=H^p(X;R^qf_*\Q)$, and tracing this action through the proof of 
 \cite[Lemma 3.1]{bibby}, one sees that it agrees with the action on the
decomposition which we will describe in a minute. 
First we recall the $W_n$-orbits in the indexing set $\Cn$. 
In Lemma \ref{lem:orbits}, we saw that for 
every $\lambda\in\QQ_n$, the set $\{\Sigma\in\Cn\ |\ \widehat{\Sigma}=\lambda\}$
is either empty or a $W_n$-orbit.
We will introduce a more convenient way to index these
orbits using labelled partitions of $q$, $\QQ_q$, so that it is independent of $n$. 

For $\lambda\in\QQ_q$, define $\lambda\langle n\rangle\in\QQ_n$ as follows: say
$(\lambda\langle n\rangle)_z = \lambda_z$ and if
$(\lambda_1,\dots,\lambda_\ell)$ are the unlabelled parts of $\lambda$ with
$\lambda_1\geq\cdots\geq\lambda_\ell>0$, let $(\lambda_1+1,\dots,\lambda_\ell+1,1,\dots,1)$ be the unlabelled parts of $\lambda\langle n\rangle$. 
In order for this to be a partition of $n$, note that we must add $n-q-\ell$
ones to the end of the partition and it will have $n-q$ unlabelled parts.
For example, if $\lambda$ is the labelled partition $(1_e,2,1)$ of 4, then $\lambda\langle 8\rangle$ is the partition $(1_e,3,2,1,1)$. 
Note that for every $\Sigma\in\Cn$ with rank $q$, there is some
$\lambda\in\QQ_q$ such that $\lambda\langle n\rangle=\widehat{\Sigma}$, so that
indexing by $\QQ_q$ will cover all of our $W_n$--orbits.

Now, the action of $W_n$ on $E_2^{pq}(n)$ permutes the summands of the decomposition according to its
action on $\Cn$. More explicitly, let $w\in W_n$. Then the action of $w$ on
$X^n$ induces maps 
$f_w:H^p(F_{\Sigma})\to H^p(F_{w^{-1}\Sigma})$ and
$g_w:H^q(M(\A_{F_{\Sigma}}))\to H^q(M(\A_{F_{w^{-1}\Sigma}}))$, so that for 
$x\otimes y$ in $H^p(F_{\Sigma})\otimes H^q(M(\A_{F_{\Sigma}}))$, we
have $w\cdot (x\otimes y) = f_{w}(x)\otimes g_{w}(y)$ in $H^p(F_{w^{-1}\Sigma})\otimes
H^q(M(\A_{F_{w^{-1}\Sigma}}))$. 
This means that if we rewrite the decomposition as
$$E_2^{pq}(n) = \bigoplus_{\lambda\in\QQ_q}\bigoplus_{\widehat{\Sigma}=\lambda\langle n \rangle} H^p(F_\Sigma)\otimes H^q(M(\A_{F_{\Sigma}})).$$
we have a decomposition into $W_n$--representations indexed by $\lambda\in\QQ_q$.
But these representations can and should be decomposed further.

Let $\Sigma\in\Cn$. 
Suppose that $\{i_1\},\{\bar{i_1}\},\dots,\{i_s\},\{\bar{i_s}\}$ are all of the
singleton parts in $\Sigma$. 
Referring to our coordinate-wise description of $F_\Sigma$ in Section
\ref{sec:components}, we see that it factors as 
$$F_\Sigma = F'_{\Sigma}\times X_{i_1}\times\cdots\times X_{i_s}$$
where each subscript $i_j$ denotes the coordinate in which the factor $X$
appears.
This means we can use the K\"unneth formula to write
$$H^p(F_\Sigma) = \bigoplus_{r+\sum a_i=p} H^r(F'_\Sigma)\otimes H^{a_1}(X_{i_1})\otimes\cdots\otimes H^{a_s}(X_{i_s}).$$
We denote $a=(a_1,\dots,a_s)$, and let $\widehat{a}\vdash(p-r)$ be the partition
which lists the nonzero elements of $a$ in decreasing order. For example, if
$a=(0,2,0,1,2)$, then $\widehat{a}=(2,2,1)$. 
Note that we may consider $a$ as an $n$-tuple where $a_1,\dots,a_s$ are the
coordinates corresponding to $i_1,\dots,i_s$ and we extend by 0.
So we have an action of $W_n$ via the action of $S_n$ by permuting coordinates.
The orbits of this action on $a$'s are indexed by $\alpha\vdash (p-r)$. 

Given $\lambda\in\QQ_q$, $r\in\{0,\dots,p\}$, and $\alpha\vdash(p-r)$, we will define the following:
$$E(\lambda,r,\alpha)_n = \bigoplus_{\widehat{\Sigma}=\lambda\langle n\rangle} \bigoplus_{\widehat{a}=\alpha} H^r(F'_\Sigma)\otimes H^{a_1}(X_{i_1})\otimes\cdots\otimes H^{a_s}(X_{i_s})\otimes H^q(M(\A_{F_\Sigma})).$$
Our claim is that each $E(\lambda,r,\alpha)_n$ is a $W_n$--representation, giving 
us the following decomposition and the indexing set in the
statement of the theorem:
$$E_2^{pq}(n) = \bigoplus_{\lambda\in\QQ_q} \bigoplus_{r=0}^p \bigoplus_{\alpha\vdash(p-r)} E(\lambda,r,\alpha)_n.$$
The action of $w\in W_n$ sends the summand indexed by $(\Sigma,r,a)$ to that
indexed by $(w^{-1}\Sigma,r,w^{-1}a)$. 
The orbit of the index $(\Sigma,r,a)$ is then indexed by $(\lambda,r,\alpha)$,
which gives the desired decomposition of $W_n$--representations, but moreover
means that $W_n$ acts transitively on the (nontrivial) summands of each
$E(\lambda,r,\alpha)_n$.

Now it remains to find the value $k$ (dependent on $(\lambda,r,\alpha)$ and
independent of $n$) and $W_k$--representation $V(\lambda,r,\alpha)$ for which
$E(\lambda,r,\alpha)_n$ is the desired induced representation.
Since $W_n$ acts transitively on the summands, for an arbitrary summand
$V(\Sigma,r,a)$ with stabilizer denoted by $G$, 
we have $E(\lambda,r,\alpha)_n = \Ind_G^{W_n} V(\Sigma,r,a)$.
We can pick a particularly nice choice of $\Sigma$ and $a$, by
``left-justifying'' in the same way Church does.
Take $a=(\alpha_1,\dots,\alpha_t,0,\dots,0)$ and $\Sigma$ to have singletons
$\{n-s+1\}$, $\{\overline{n-s+1}\}$, $\dots$, $\{n\}$, $\{\bar{n}\}$ along with
some fixed partition of \textbf{n-s} (independent of $n$ since
$n-s=q+\ell$). 

Let $\ell$ be the number of unlabelled parts of $\lambda$ and $\ell(\alpha)=t$,
and define $k=q+\ell+t$. This is independent of $n$, but note that $k=n-s+t$.
Consider $W_k$ as the subgroup of $W_n$ which acts on $\textbf{k}$, and consider $W_{n-k}$ as acting on $\n\setminus \textbf{k}$.
The stabilizer $G$ of our summand $V(\Sigma,r,\alpha)$ satisfies $W_{n-k}\subseteq G\subseteq W_k\times W_{n-k}$, and moreover, $W_{n-k}$ acts trivially on $V(\Sigma,r,a)$.
Thus, we can write $G=H\times W_{n-k}$ for some $H\subseteq W_k$ and view $V(\Sigma,r,a)$ as a representation over $H$. We define $V(\lambda,r,\alpha) = \Ind_H^{W_k} V(\Sigma,r,a)$.
Note that by our choice of $\Sigma$ and $a$, the $W_k$--representation $V(\lambda,r,\alpha)$ does not depend on $n$.

Finally,
\begin{align*}
E(\lambda,r,\alpha)_n &= \Ind_G^{W_n} V(\Sigma,r,a)\\
 &= \Ind_{H\times W_{n-k}}^{W_n}  V(\Sigma,r,a)\boxtimes\Q\\
 &= \Ind_{W_k\times W_{n-k}}^{W_n} \left( \Ind_{H\times W_{n-k}}^{W_k\times W_{n-k}} V(\Sigma,r,a)\boxtimes\Q \right) \\
 &= \Ind_{W_k\times W_{n-k}}^{W_n} \left( \Ind_H^{W_k} V(\Sigma,r,a) \right) \boxtimes\Q\\
 &= \Ind_{W_k\times W_{n-k}}^{W_n} V(\lambda,r,\alpha)\boxtimes\Q.
\end{align*}
\end{proof}

\begin{remark}
While the lemma required $n\geq p+2q$, one could still make sense of such a
decomposition for any $n$. The point of making a restriction on $n$ is so that
the partitions actually determine representations of $W_n$. But one could just
ignore the summands which don't make sense because $n$ is too small. 
\end{remark}

\begin{theorem}\label{thm:repstable}
Let $X$ be one of $\C$, $\C^\times$, or a complex elliptic curve, and let $W_n$
be the hyperoctahedral group. Let $\{\Phi_n\}$ be a sequence of root systems in
type B, C, or D, and let $\An=\A(X,\Phi_n^+)$ be the corresponding arrangements
in $X^n$. Then for each $i\geq0$, the sequence $\{H^i(M(\An))\}$ of
$W_n$-representations is uniformly representation stable with stable range
$n\geq 4i$. 
\end{theorem}
\begin{proof}
Note that for any inclusion $\iota:\n\into \textbf{m}$ with
$\iota(\bar{k})=\overline{\iota(k)}$, we have induced maps $X^m\to X^n$ and
$M(\A_m)\to M(\An)$. By functoriality of the Leray spectral sequence, we have
maps from the Leray spectral sequence associated to $\An$ to that of $\A_m$.
This makes the Leray spectral sequence a spectral sequence of FI$_W$-modules.
We claim that our decomposition of $E_2^{pq}(n)$ in Lemma \ref{lem:E2} actually
gives us a decomposition of FI$_W$-modules
$$E_2^{pq} = \bigoplus_{I} \M_W(V(\lambda, r, \alpha)).$$
Since each summand on the right-hand side is finitely generated with weight
$\leq p+2q$ and stability degree $\leq p+2q$, the finite direct sum giving
$E_2^{pq}$ must be as well. 
In each of our cases, we have $E_3=E_\infty$ (see Lemma 3.2 and Remark 3.5 in
\cite{bibby}), so the theorem then follows from Proposition
\ref{prop:ssfiwmods} and Theorem \ref{thm:fgstable}.

To see our decomposition as FI$_W$-modules, the fact that $V(\lambda, r,\alpha)$
does not depend on $n$ tells us that for each morphism (ie, injection)
$\iota:\n\to\textbf{n+1}$ in the category FI$_W$, the following diagram commutes:
\begin{center}
\begin{tikzpicture}
\node (0) at (-3,0) {$\bigoplus\Ind_{W_k\times W_{n-k}}^{W_n} V(\lambda,r,\alpha)\boxtimes\Q$};
\node (1) at (3,0) {$\bigoplus\Ind_{W_k\times W_{n+1-k}}^{W_{n+1}} V(\lambda,r,\alpha)\boxtimes\Q$};
\node (2) at (-2,2) {$E_2^{pq}(n)$};
\node (3) at (2,2) {$E_2^{pq}(n+1)$};
\draw[<->] (2) -- (-2,.5);
\draw[<->] (3) -- (2,.5);
\draw[->] (2) -- (3);
\draw[->] (0) -- (1);
\end{tikzpicture}
\end{center}
\end{proof}

\begin{remark}
The argument in Lemma \ref{lem:E2} is very similar to that given by Church \cite{church} in the type A case, generalized to work with $W_n$ and labelled partitions (rather than $S_n$ and partitions). 
In a separate paper, Church, Ellenberg, and Farb \cite[Theorem 6.2.1]{cef} provide an alternative proof of representation stability for the type A case, which uses the fact that the $E_2$-term is generated by the cohomology of the linear arrangement along with the cohomology of the ambient space. 
The fact that these generators are finitely generated FI-modules is enough to show that each piece of the $E_2$-term is. 
However, in the other cases, the lack of unimodularity makes it more complicated. We still have that the cohomology of the linear arrangement and the cohomology of the ambient space give finitely generated \fiw-modules \cite{wilson1}. However, these together are not enough to generate the $E_2$-term. 
Instead of dealing with these extra generators separately, we have decided to follow Church's original argument more closely. 
\end{remark}

There are a few easy consequences of the stability. 
First, if we consider $\W_n$ to be the type A or type D Weyl groups, both of
which are a subgroup of the type B/C Weyl group $W_n$, then the restriction of
an FI$_W$-module to an \fiw-module preserves finite generation
\cite[Prop.~3.22]{wilson2}. This gives us the following:
\begin{corollary}\label{cor:restrict}
Let $\{\Phi_n\}$ be a sequence of root systems in type B, C, or D.
Let $X$ be one of $\C$, $\C^\times$, or a complex elliptic curve, and let 
$\An=\A(X,\Phi_n^+)$ be the corresponding arrangement in $X^n$. 
Let $\W_n$ be the type A, B/C, or D Weyl groups. Then the sequence
$\{H^i(M(\An))\}$ of $\W_n$--representations is uniformly representation
stable.
\end{corollary}

Because of this, we will state the other corollaries in this generality.
We also obtain an analogue of \cite[Cor. 5.10]{wilson1} on the polynomiality of characters. 

\begin{corollary}\label{cor:polychar}
Let $\{\Phi_n\}$ be a sequence of root systems in type A, B, C, or D, and let
$\W_n$ be the corresponding Weyl groups. Let $X$ be one of $\C$, $\C^\times$, or
a complex elliptic curve, and let $\An=\A(X,\Phi_n^+)$ be the corresponding
arrangement in $X^n$. 
Then the sequence of characters of the $\W_n$-representations $H^i(M(\An))$ are 
given by a unique character polynomial of degree $\leq 2i$.
In particular, we have that $\dim H^i(M(\An))$ is a polynomial in $n$ of degree $\leq 2i$. 
\end{corollary}

Finally, since $H^i(M(\An)/\W_n)\cong H^i(M(\An))^{\W_n}$ and Theorem \ref{thm:repstable} implies stability of the multiplicity of the trivial representation, we can make a statement on homological stability. 
Arnol'd \cite{arnold} established homological stability for the type A linear
arrangement, and Church \cite{church} gave a more general type A homological
stability result. The other linear cases have been studied by 
Brieskorn \cite{brieskorn}, but as far as the author knows, 
it has not been stated for the toric and elliptic analogues in types B, C, and D.

\begin{corollary}\label{cor:homstab}
Let $\{\Phi_n\}$ be a sequence of root systems in type A, B, C, or D, and let
$\W_n$ be the corresponding Weyl groups. Let $X$ be one of $\C$, $\C^\times$, or
a complex elliptic curve, and let $\An=\A(X,\Phi_n^+)$ be the corresponding
arrangement in $X^n$. Then the orbit spaces $M(\An)/\W_n$ enjoy rational
homological stability. That is, for each $i$, $H^i(M(\An)/\W_n)$ does not
depend on $n$ for $n\geq 4i$.
\end{corollary}

\subsection{Examples and Computations}
Computations such as finding the stable multiplicities of irreducible representations and the character polynomials are difficult in general. 
In \cite[Thm.~1(1)]{chen}, Chen recently gave a generating function for the
stable multiplicities of the representations
$H^i(M(\A(X,\Phi_n^+)))$ in the case that $X=\C$ and $\Phi_n$ is type A.
But more general computations, even for the other linear or type A cases, are
not known. 

One aspect of the elliptic case that might make it harder is that not even the Betti numbers are known in general (there is a nice combinatorial description of the Betti numbers for linear and toric arrangements). 
But if one wanted to compute the stable multiplicities for the elliptic case, one might try to first compute them for the $E_2$-term. What you see, even in type A, is some tensor products of the linear case with an exterior algebra. Thus, even if the multiplicities of the linear case were known, one would have to deal with computation of the Kronecker coefficients from the tensor product. 

We do show some work for the degree one cohomology. Even in degree two, though,
 it starts to get more complicated. 

\begin{example}
Here are computations of the stable multiplicities of $H^1(M(\An))$ when
$\Phi_n$ is type A and $\An=\A(X,\Phi_n^+)$.
\begin{enumerate}
\item If $X=\C$, then $H^1(M(\An))=V(0)\oplus V(1)\oplus V(2)$ for $n\geq 4$. Church and Farb give this and a decomposition for degree two in \cite{churchfarb}.
\item If $X=\C^\times$, then $H^1(M(\An))=V(0)^{\oplus2}\oplus V(1)^{\oplus 2}\oplus V(2)$ for $n\geq 4$.
\item If $X$ is an elliptic curve, then we have
$E_\infty^{01}(n)=0$ and hence for $n\geq2$, 

$H^1(M(\An))=E_2^{10}(n)=V(0)^{\oplus2}\oplus V(1)^{\oplus2}$. 
\end{enumerate}
\end{example}

\begin{example}\label{ex:decomp}
In this example, we demonstrate the decomposition of Lemma \ref{lem:E2} for
first-degree cohomology $H^1(M(\A(\C^\times,\Phi^+_2)))$ in the case that
$\Phi_2$ is type B. 
In the Leray spectral sequence for toric arrangements, we have that
$E_2=E_\infty$, and so the decompositions of $E_2^{01}$ and $E_2^{10}$ together
give a decomposition of the cohomology. 
First, we consider our poset of labelled partitions, drawn so that orbits are
grouped together. 
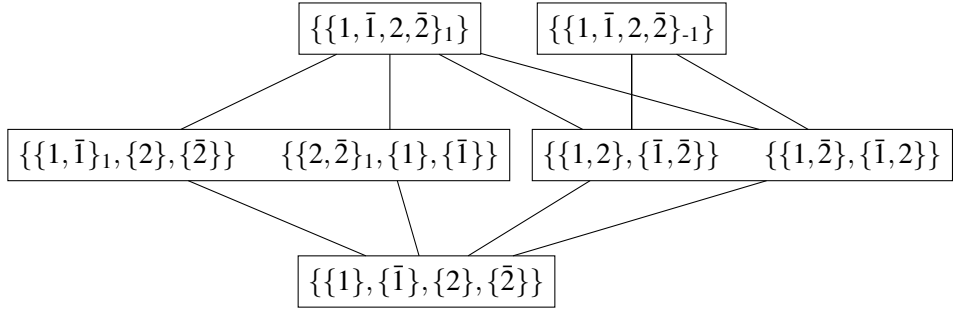
\begin{center}
\begin{tikzpicture}%[scale=0.45]
\node (0) at (0,0) {$\{\{1\},\{\bar{1}\},\{2\},\{\bar{2}\}\}$};
\node[above right= 1cm and -3mm of 0] (4)  {$\{\{1,2\},\{\bar{1},\bar{2}\}\}$};
\node[right= 3mm of 4] (5)   {$\{\{1,\bar{2}\},\{\bar{1},2\}\}$};
\node[left=3mm of 4] (3)  {$\{\{2,\bar{2}\}_1,\{ 1\},\{\bar{1}\}\}$};
\node[left=3mm of 3]  (2)  {$\{\{1,\bar{1}\}_1,\{ 2\},\{\bar{2}\}\}$};
\node[above=of 4] (C)  {$\{\{1,\bar{1},2,\bar{2}\}_{\text{-1}}\}$};
\node[above=of 3] (B)  {$\{\{1,\bar{1},2,\bar{2}\}_1\}$};
\draw[-] (0)--(2)--(B)--(3)--(0)--(4)--(B)--(5);
\draw[-] (C)--(4)--(C)--(5)--(0);
\draw (4.south west) rectangle (5.north east);
\draw (0.south west) rectangle (0.north east);
\draw (2.south west) rectangle (3.north east);
\draw (B.south west) rectangle (B.north east);
\draw (C.south west) rectangle (C.north east);
\end{tikzpicture}
\end{center}

Even though this example is about $n=2$, there is one note that we can make here
for general $n$. We have $E_2^{10}=H^1((\C^\times)^n)$, which is the first
degree part of an exterior algebra of $\Q^n$. The Weyl group acts in the
standard way on $\Q^n$, giving us $E_2^{10}(n) = V((n-1),(1))$. This tells us
in particular that 
$E_2^{10}(2)=V((1),(1))$, and stably we have 
$E_2^{10}(n)=V(\emptyset,(1))_n$. 

For $E_2^{01}$, we must have $r=0$ and $\alpha=0$, and so our decomposition is indexed by the two orbits in the middle of the above picture. These correspond to the two labelled partitions of 1: $(1_1)$ and $(1)$. 
In the first orbit, let $\Sigma = \{\{1,\bar{1}\}_1,\{2\},\{\bar{2}\}\}$, and in
the second orbit, let $\Sigma'=\{\{1,2\},\{\bar{1},\bar{2}\}\}$.
We have:
\begin{align*}
H^1(M(\A_2)) &= E_2^{10}(2) \oplus E_2^{01}(2)\\
&= V((1),(1))\oplus E((1_1),0,0)_2 \oplus E((1),0,0)_2\\
&= V((1),(1))\oplus \Ind_{W_1\times W_1}^{W_2} V(\Sigma,0,0) \oplus
\Ind_{D_2}^{W_2} V(\Sigma',0,0)\\
&= V((1),(1)) \oplus \Ind_{W_1\times W_1}^{W_2} H^1(M(\A_{F_\Sigma})) \oplus
\Ind_{D_2}^{W_2} H^1(M(\A_{F_{\Sigma'}}))\\
&= V((1),(1)) \oplus \Ind_{W_1\times W_1}^{W_2} H^1(\C^2\setminus H_1) \oplus
\Ind_{D_2}^{W_2} H^1(\C^2\setminus H_{12})\\
&= V((1),(1)) \oplus \Ind_{W_1\times W_1}^{W_2} V((1),\emptyset)\boxtimes\Q \oplus \Ind_{D_2}^{W_2}
V((2),\emptyset)
\end{align*}

Note that this does not give us the stable multiplicities, since we need $n\geq
4$. However, 
Wilson \cite{wilson1} gave a decomposition of $E_2^{01}$, which we can consider
as the first degree cohomology of the linear type B/C arrangement. This
decomposition, as an FI$_{W}$-module is $\M_{W}(1,\emptyset)\oplus
\M_{W}(2,\emptyset)\oplus \M_{W}(\emptyset,2)$. 
By decomposing this into irreducibles when $n=4$ and using
$E_2^{10}(n)=V(\emptyset,(1))_n$, one could compute the stable multiplicities.

If we had considered the type D$_2$ arrangement, we would have $$H^1(M(\A_2))
= V((1),(1))\oplus V((2),\emptyset)\oplus V(\emptyset,(2)).$$
The only difference from type B$_2$ is that we have only one orbit of rank one, indexed by the partition $(1)$. 

If we had considered the type C$_2$ arrangement, we would have three orbits: $(1_1)$, $(1_{-1})$, and $(1)$. The first and last would act as before; the new orbit would act just as $(1_1)$ did. 
Thus, we would have the same decomposition as in the type B$_2$ case with an
extra factor of $\M_{W}(1,\emptyset)$. 
\end{example} 

\begin{example}
In this example, we demonstrate an aspect of the polynomiality of characters as in Corollary \ref{cor:polychar}. 
For each of our arrangements $\An=\A(X,\Phi_n^+)$, we state the dimension of
$H^1(M(\An))$. These formulas hold for all $n\geq 2$. 

\begin{center}
\renewcommand{\arraystretch}{2}
\begin{tabular}{|c|c|c|c|}
\hline
 & $X=\C$ & $X=\C^\times$ & $X=E$ \\
\hline
Type A$_{n-1}$ & $\binom{n}{2}$ & $\binom{n}{2}+n$ & $2n$ \\
\hline
Type B$_n$ & $2\binom{n}{2}+n$ & $2\binom{n}{2}+2n$ & $\binom{n}{2}+2n$ \\
\hline
Type C$_n$ & $2\binom{n}{2}+n$ & $2\binom{n}{2}+3n$ & $\binom{n}{2}+5n$ \\
\hline
Type D$_n$ & $2\binom{n}{2}$ & $2\binom{n}{2}+n$ & $2n$ \\
\hline
\end{tabular}
\end{center}
\end{example}

\subsection{An improvement for type A}
In our main results, we had ignored type A for ease of working only with $W_n$ and because the result was already known, but the stable range for some type A arrangements can be improved. 
Recall that we know the sequence stabilizes once $n\geq 4i$, in each type (that is, for each $X$ and family of root systems). 
Recently, Hersh and Reiner \cite{hershreiner} improved the stable range for the type A linear case, showing that the $i$-th cohomology stabilizes for $n\geq 3i+1$. 
We show an improvement for the elliptic case, and we wonder if it can be improved further, or if Hersh and Reiner's result can be used to improve the range of the toric case. 
\begin{proposition}\label{prop:aelliptic}
If $\{\Phi_n\}$ is a sequence of type A root systems and $X$ is a complex
elliptic curve, let $\An=\A(X,\Phi_n^+)$. Then 
for each $i\geq1$, the stable range of the sequence $\{H^i(M(\An))\}$ of
$S_n$--representations may be improved to $n\geq 4i-2$. 
\end{proposition}
\begin{proof}
Fix $i\geq 1$.
We claim that the differential $d:E_2^{0,i}(n)\to E_2^{2,i-1}(n)$ is injective
for all $n$, and hence $E_3^{0,i}=E_\infty^{0,i}=0$ for all $n$. 
Thus, in our filtration $F_0\subseteq\cdots\subseteq F_i=H^i$, the maximum weight (and similarly stability degree) among $F_j$ and $F_j/F_{j-1}$ is $2i-1$. This implies that $H^i(M(\An))$ is representation stable for $n\geq 2(2i-1)=4i-2$. 

First, we refer the reader to \cite{totaro} for a description of a decomposition
of and differential on the $E_2$ term, as well as \cite{bez} for a description
of a basis, in the case of type A arrangements.
Now to show injectivity of this differential, we pick our standard generators $x_i$, $y_i$ ($1\leq i\leq n$) for $H^*(X^n)$ and $g_{ij}$ ($1\leq i<j\leq n$) for $$E_2^{0,1}(n)=\displaystyle\bigoplus_{1\leq i<j\leq n} H^1(M(\A_{H_{ij}})).$$
The differential sends $g_{ij}$ to the class of the diagonal $H_{ij}$ in $H^2(X^n)$, which is given by $(x_i-x_j)(y_i-y_j)$. 
A basis for $E_2^{2,0}=H^2(X^n)$ is given by pairs $x_ix_j$, $y_iy_j$ ($i\neq j$) along with pairs $x_iy_j$. 
So if we had a linear combination $\sum c_{ij}g_{ij}$ in the kernel of the differential, then its image
$\sum c_{ij}(x_i-x_j)(y_i-y_j)$ would be zero. We can write this sum in terms of the basis as $\sum d_{i}x_iy_i - \sum c_{ij}(x_iy_j + x_jy_i)$ for some coefficients $d_i$.
Since $x_iy_j$ appears once, with coefficient $c_{ij}$, we must have $c_{ij}=0$. 

To extend this to $E_2^{0,q}\to E_2^{2,q-1}$, we recall the basis given by
Bezrukavnikov \cite{bez}. For
$E_2^{0,q}$, we have monomials $g_{i_1j_1}\dots g_{i_qj_q}$ with $i_s>j_s$
for each $s$ and with $i_1>i_2>\cdots >i_q$. Similarly, a basis for
$E_2^{2,q-1}$ is given by $z_az_bg_{i_1j_1}\dots g_{i_{q-1}j_{q-1}}$ where we
have the same conditions on the indices on $g$'s, $z_a$ stands for $x_a$ or
$y_a$, $z_b$ stands for $x_b$ or $y_b$, and we have
$a,b\notin\{i_1,\dots,i_{q-1}\}$.
Then suppose we have $d(\sum c_Sg_S)=0$ where $S$ runs through the sets indexing
our monomial basis and the $c_S$ are some rational coefficients. 
Consider one such $S=\{(i_1,j_1),\dots,(i_q,j_q)\}$. In the expansion of $d(\sum
c_Sg_S)$ in the monomial basis, the coefficient of
$x_{i_q}y_{j_q}g_{i_1j_1}\dots g_{i_{q-1}j_{q-1}}$ is equal to $\pm c_S$. This
is because we have $i_1>i_2>\cdots>i_{q-1}>i_q>j_q$; there is no other $S'$ from which
this monomial arises in the differential. Thus, we must have each $c_S=0$.
\end{proof}

\bigskip
\textbf{Acknowledgements. }
 The author wishes to thank Nick Proudfoot and Benson Farb for piquing her interest in this project, and Graham Denham for helpful conversations.
Many thanks to Jenny Wilson for pointing out and helping to fix an error in an earlier draft. 
Finally, the author is grateful to the referees for many useful comments and
suggestions.

\end{document}